\documentclass[12pt, reqno]{article}
\usepackage{amsmath, amsthm, amscd, amsfonts, amssymb, graphicx, color}
\usepackage[bookmarksnumbered,  plainpages]{hyperref}
\usepackage{setspace,lineno}

\title{Maps preserving ascent/descent of triple Jordan product}

\author{Roja Hosseinzadeh$^{1}$ and Tatjana Petek$^{2,3}$\footnote{Corresponding author, email: tatjana.petek@um.si} \\
	$^1$	{\it Department of Mathematics, Faculty of Mathematical Sciences,} \\
	{\it University of Mazandaran, P.O.Box 47416-1468, Babolsar, Iran}\\
	$^2${\it Faculty of Electrical Engineering and Computer Science,} \\ {\it University of Maribor,
		Koro\v ska cesta 46, SI-2000 Maribor, Slovenia}\\
	$^3${\it Institute of Mathematics, Physics and Mechanics,}\\
	{\it Jadranska 19, SI-1000 Ljubljana, Slovenia} 
}

\newtheorem{theorem}{Theorem}[section]
\newtheorem{lemma}[theorem]{Lemma}
\newtheorem{proposition}[theorem]{Proposition}
\newtheorem{corollary}[theorem]{Corollary}

\theoremstyle{definition}

\newtheorem{example}[theorem]{Example}

\theoremstyle{remark}
\newtheorem{remark}[theorem]{Remark}
\numberwithin{equation}{section}

\newcommand{\F}{\mathbb{F}}

\newcommand{\bx}{\mathcal{B(X)}}
\newcommand{\rk}{\mathrm{rank}\,}

\begin{document}
	\maketitle
	
	
	
	\begin{abstract}
		Let $\mathcal{X}$ be a real or complex Banach space with $ \dim \mathcal{X}\geq 3$. We give a complete description of surjective mappings on $\mathcal{B(X)}$ that preserve the ascent of Jordan triple product of operators or, preserve the descent of Jordan triple product of operators.
	\end{abstract}

	\section{Introduction}
	
	Preserving problems in operator theory in recent decades are a topic of interest to many mathematicians. In these problems, the authors are interested
	in describing mappings defined on operator algebras 
	which leave some functions, subsets, relations, etc. fixed. Many results of
	this type can be found in papers \cite{3}, \cite{4}, \cite{10}, \cite{11} and references therein.
	
	Let $\bx$ denote the algebra of all bounded linear operators on a complex or real Banach space $\mathcal{X}$ and let $\mathbb{F}$ denote the field of scalars of $\mathcal{X}$ with $\mathbb{F}^\ast=\mathbb{F}\setminus \!\left\{0\right\}$. For an operator $T \in \mathcal{B(X)}$, the ascent $\alpha (T)$ and the descent $ \delta (T)$ are given by
	
	$$\alpha (T)=\inf \{ n \in \mathbb{N}\cup \{0\}:~ \ker (T^n)= \ker (T^{n+1})\},$$
	$$ \delta (T)=\inf \{ n \in \mathbb{N}\cup \{0\}:~ \mathrm{range} (T^n)= \mathrm{range} (T^{n+1})\}.$$
	We set $\alpha(T)=\infty$ and $\delta(T)=\infty$, respectively, when the relevant infimum does not exist. The indices $\alpha (T)$ and $ \delta (T)$ were introduced by Riesz \cite{12}. These notions have been used as tools in the study of several spectral properties of some classes of linear operators in Banach spaces, see for instance \cite{7a} and the references therein. To learn basic facts about these values see \cite{1}, \cite{7} and \cite{8}.

	Authors in \cite{2} proved that a surjective additive map $\varphi: \mathcal{B(X)} \rightarrow \mathcal{B(Y)}$ where $\mathcal{X}$  and $\mathcal{Y}$ are infinite-dimensional Banach spaces, preserves the ascent of operators if and only if there exist a nonzero complex number $\lambda$ and an invertible bounded linear or conjugate linear operator $A: \mathcal{X} \rightarrow \mathcal{Y}$ such that $\varphi (T) = \lambda ATA^{-1}$ for all $T \in \mathcal{B(X)}$.
	Moreover, they obtained a similar result for additive maps from $\bx$ onto $\mathcal{B(Y)}$ preserving the descent of operators.
	
	In the field of so-called preserver problems, some researchers consider the assumption of preserving a specific property of a fixed product of operators instead of the assumption of preserving the specific property of operators. In this type of preserver problem, the linearity of the mapping is usually removed. In this regard, in \cite{5} we characterized  maps on $ \mathcal{B(X)}$ which preserve the ascent of product of operators or, they preserve the descent of product of operators. In fact, in \cite{5} it is proved that $\phi: \mathcal{B(X)} \rightarrow \mathcal{B(X)}$  is surjective map with $\phi(I)$ surjective and $\alpha(AB)=\alpha(\phi(A)\phi(B))$ (or $\delta(AB)=\delta(\phi(A)\phi(B))$) for every $A,B\in\bx$ (only when  $\mathcal{X}$ is infinite-dimensional) if and only if
	there exist an invertible bounded linear or conjugate-linear operator $T: \mathcal{X} \rightarrow \mathcal{X}$ and a function $k:\mathcal{B(X)} \rightarrow \mathbb{F}^*$ such that $\phi (A) = k(A)TAT^{-1}$ for all $A \in \mathcal{B(X)}$.
	Also, in the case where the dimension of $\mathcal{X}$ is finite and greater than or equal to $3$, a similar description was obtained, but here without the condition "$\phi(I)$ is surjective/injective". It turned out that both problems are connected with preservers of the rank-one nilpotency of the product.
	
	In this paper, we are going to study a similar problem for maps preserving the ascent of the triple Jordan product or, alternatively and simultaneously, preserving the descent of the triple Jordan product of operators. It turns out that we have to alter the proof working for the usual product quite a lot in order to serve our needs. Below is our main result.

	\medskip

	\begin{theorem}\label{th:1} Let $\mathcal{X}$ be at least three-dimensional Banach space over the field $\F$, being either the field of complex or the field of all real numbers.  Let $\phi:\bx \to \bx$ be a surjective map satisfying the condition
		\begin{equation}
			\alpha(ABA)=\alpha(\phi(A)\phi(B)\phi(A)),\ \ \ \text{for every}\ \ A,B\in\bx,
		\end{equation}
		or,
		\begin{equation}
			\delta(ABA)=\delta(\phi(A)\phi(B)\phi(A)),\ \ \ \text{for every}\ \ A,B\in\bx.
		\end{equation}
		If $\mathcal{X}$ is infinite-dimensional space, then  either there exists an invertible bounded linear or conjugate-linear operator $A: \mathcal{X}\to \mathcal{X}$ and a map  $\lambda :\bx \to \F^\ast$,    such that
		\[\phi(T)=\lambda(T) ATA^{-1},\ \ \ T\in  \bx, \]
		or $\mathcal{X}$ is reflexive and,
		there exists an invertible bounded linear or conjugate-linear operator $A: \mathcal{X}^\ast\to \mathcal{X}$ and a  map  $\lambda :\bx \to \F^\ast$ such that
		\[\phi(T)=\lambda(T) AT^\prime A^{-1},\ \ \ T\in  \bx. \]
		When $\mathcal{X}$ is of finite dimension $n\ge 3$, we identify operators with n-by-n matrices. Then  there exists a nonsingular matrix $A$, a field automorphism $\tau$ of $ \F$ and a map   $\lambda :M_n(\F) \to \F^\ast$ such that $\phi$ is either of the form
		\begin{equation}\label{eq:1000}
			\phi([t_{ij}]) = \lambda(T) A[\tau(t_{ij})]A^{-1},\ \ \ [t_{ij}]\in M_n(\F),
		\end{equation}
		or, it  is of the form
		\begin{equation}\label{eq:2000}
			\phi([t_{ij}]) = \lambda(T) A[\tau(t_{ij})]^{\mathrm{tr}}A^{-1},\ \ \ [t_{ij}]\in M_n(\F),
		\end{equation}
		where $.^\mathrm{tr}$ denotes the transposition.
	\end{theorem}
	\begin{remark}
		It is well known that the only automorphism of $\mathbb{R}$ is the identity, while there are many "wild" automorphisms of $\mathbb{C}$. However, the only continuous automorphisms of $\mathbb{C}$ are the identity and the complex conjugation.
	\end{remark}
	We provide examples showing that the assumption $n\ge 3$ and surjectivity in the infinite-dimensional case are indispensable.
	\begin{example}
		Let $\mathcal{X}$ be an infinite dimensional Banach space. Then as it is well known, it is isomorphic to $\mathcal{X}\oplus \mathcal{X}$. Let $T:\mathcal{X}\oplus \mathcal{X}\to \mathcal{X}$ be such an isomorphism. Note that every operator in $\mathcal{
			B}(\mathcal{X}\oplus \mathcal{X})$ can be represented by a two-by-two operator matrix.  Then the map $\phi:\mathcal{
			B}(\mathcal{X}\oplus \mathcal{X})\to \mathcal{
			B}(\mathcal{X}\oplus \mathcal{X})$, defined by
		$$\phi(A) = \begin{bmatrix}
			TAT^{-1} & 0 \\
			0 & I
		\end{bmatrix},\ \ \ A\in\mathcal{
			B}(\mathcal{X}\oplus \mathcal{X}),$$
		is multiplicative, obviously has the property that $\alpha(A)=\alpha(\phi(A))$ and $\delta(A)=\delta(\phi(A))$ for every $A$ and so, it preserves the ascent and the descent of Jordan triple products.
	\end{example}

	Our paper is organized as follows. In section Preliminaries we give the notation and several technical lemmas further applied in the proof. The proof itself is given in the last section.
	\section{Preliminaries}
	We start with the notations.
	The symbol $\mathcal{X}$ stands for a complex or real Banach space (over $\mathbb{F}$ correspondingly) of dimension at least three and $\bx$ denotes the algebra of all bounded linear operators on $\mathcal{X}$.  The dual space of $\mathcal{X}$ is denoted by $\mathcal{X}^ *$ and $A^{\prime}$ is the adjoint operator of $A \in \bx$.  By a functional we always mean a bounded functional.  Both, the identity operator on $\mathcal{X}$ and the identity matrix in $M_n(\F)$, the algebra of all $n\times n$ matrices, will be denoted by $I$.
	For every nonzero $x\in \mathcal{X}$ and nonzero
	$f\in \mathcal{X}^ *$, the symbol $x\otimes f$ stands for the rank-one
	linear operator on $\mathcal{X}$ defined by $(x\otimes f)y=f(y)x$ for any
	$y\in \mathcal{X}$. Note that every rank-one operator in $\bx$ can be
	written in this way. The operator $x\otimes f \neq 0$ is
	idempotent if and only if $f(x)=1$ and is nilpotent if and only if $f(x)=0$. We denote by  $\mathcal{P}_1(\mathcal{X}) $ and $\mathcal{N}_1(\mathcal{X})$  the
	set of all rank-one idempotent operators and the set of all rank-one nilpotent operators in $\bx$, respectively. For a subset $M\subseteq \mathcal{X}$ we denote $M^\circ =\{f\in\mathcal{X}^\ast; f(m)=0, \; m\in M\}$. By $E_{ij}\in M_n(\F)$ we denote the standard basis matrix in $M_n(\F)$ having $1$ in the $(i,j)$-position and zeros elsewhere and by $\mathrm{diag}(a_1,a_2,\dots,a_n)$ we mean a diagonal matrix with the given diagonal entries. 
	
	Listing the basic properties of the ascent and descent, which can be easily verified and the proof will therefore be omitted,  are in order. 
	
	\begin{lemma}\label{lem: basic}
		For every operator $A\in \bx$ the following assertions hold.
		\begin{enumerate}
			\item If $A$ is injective if and only if $\alpha(A)=0$.
			\item If $A$ is surjective if and only if $\delta(A)=0$.
			\item If $A$ is nilpotent of nilindex $k$, then $\alpha(A)=\delta(A)=k$.
			\item If $A^2=A$ and $A\neq I$, we have $\alpha(A)=\delta(A)=1$; in particular, $\alpha(0)=\delta(0)=1$.
			\item If $A$ is algebraic with the minimal polynomial $m_A(\lambda)=\lambda^2(\lambda -\lambda_0)$, $\lambda_0 \ne 0$, then  $\alpha(A)=\delta(A)=2$.
		\end{enumerate}
		
	\end{lemma}

	We next present several pairs of matrices regarding their ascent/descent/Jordan-triple product properties which will be further applied in the proof.
	For finding particular matrices in the following lemmas, we used Wolfram Mathematica tools which application we declare in Aknowledgement at the end of the paper.
	
	\begin{lemma}\label{lem:list}
		Suppose $a,b\in \F\setminus \{1\}$ and $a\ne b$. Then the  matrices 
		\begin{equation}\label{eq:3}
			C_a= \begin{bmatrix}
				1 & 0 & 0 \\
				0 & 1 & 0\\
				0 & 0 & a 	
			\end{bmatrix}, \   C_b= \begin{bmatrix}
				1 & 0 & 0 \\
				0 & 1 & 0\\
				0 & 0 & b 	
			\end{bmatrix}, \  T_a=\begin{bmatrix}
				1 & 1 & 1 \\
				\frac{1}{a-1}	& 0 & 0 \\
				\frac{-1}{a-1} & 0 & 0
			\end{bmatrix}
		\end{equation}
		satisfy $\alpha(T_aC_aT_a)=\delta(T_aC_aT_a)=3$,  $\alpha(T_aC_bT_a)=\delta(T_aC_bT_a)=2$ and $T_a$ is algebraic of order three. 
	\end{lemma}
	
	\begin{proof}
		First of all, observe that the minimal polynomial of $T_a$ is equal to $\lambda^2(\lambda -1)$ validating the last statement. By a direct computation we get that $T_aC_aT_a$ is nilpotent of nilindex three and by the property (3) in Lemma \ref{lem: basic} the ascent and the descent are equal to three.  The matrix $T_aC_aT_a$ has the minimal polynomial $\lambda^2(\lambda-\frac{a-b}{a-1})$ and so, by (5) in Lemma \ref{lem: basic}, both the ascent and the descent of $T_aC_aT_a$ are equal to $2$.
	\end{proof}
	\begin{lemma}\label{lem:12}
		Let $u,v\in \F^\ast$, $u\neq v$.
		\begin{equation}\label{eq:4}
			A_0= \begin{pmatrix}
				1 & u & 0 \\
				0 & 1 & 0\\
				0 & 0 & 1 	
			\end{pmatrix}, \ \ \  B_0= \begin{pmatrix}
				1 & v & 0 \\
				0 & 1 & 0\\
				0 & 0 & 1 	
			\end{pmatrix}, \ \  	T=\begin{bmatrix}
				-2 u & 0 & 0 \\
				1 & 0 & 0 \\
				\frac{1}{2 (v-u)}	 & 1 & 0
			\end{bmatrix}.
		\end{equation}
		Then   $\alpha(A_0TA_0)=\delta(A_0TA_0) =3$, $\alpha(B_0TB_0)= \delta(B_0TB_0)=2$ and $T$ is algebraic of order three.
	\end{lemma}
	\begin{proof}
		Check that the minimal polynomial of $T$ reads $m_T(\lambda)=\lambda^2(\lambda +2u)$ and compute that $(A_0TA_0)^3$ is nilpotent of nilindex three thus by (3) in Lemma \ref{lem: basic}, the ascent and the descent are equal to three. The minimal polynomial of $B_0TB_0$ is $\lambda^2(\lambda -2(v-u))$. Since $u\neq v$, the ascent and the descent are equal $2$ due to (5) in Lemma \ref{lem: basic}.
	\end{proof}
	\begin{lemma}\label{lem:125}
		Let 
		\begin{equation}\label{eq:125}
			A=\begin{bmatrix}
				1&0&0\\
				0&1&0\\
				0&0&0
			\end{bmatrix},\ \ \ B=\begin{bmatrix}
				0&0&0\\
				0&0&0\\
				0&0&-1
			\end{bmatrix}, \ \ \ T=\begin{bmatrix}
				1&1&0\\
				1&0&1\\
				-1&0&0
			\end{bmatrix}.
		\end{equation}
		Then $T$ has minimal polynomial of degree three and  $\alpha(TAT)=\delta(TAT)=1$ and $\alpha(TBT)=\delta(TBT)=2$.
	\end{lemma}
	\begin{proof}
		From direct computation it follows that the minimal polynomial of $T0$ equals $(\lambda +1)(\lambda -1)^2$  and, for some invertible matrix $S\in M_3(\F)$ we have
		$TAT = S\mathrm{diag}(0,1,2)S^{-1}$ and $TBT=E_{21}$. The result then easily follows.
	\end{proof}
	
	\begin{lemma} \label{lem:13}
		Assume $a$, $b$, $w\in \F$, $w\neq 0$ and, let
		\begin{align} \label{al:1}
			A(a,b)&=\begin{bmatrix}
				0& b& 0 & 0 \\
				1& a & 0& 0\\
				0 &0& 0 & b \\
				0 & 0& 1 & a
			\end{bmatrix},  \  
			B(a,b,w)= wI + A(a,b), \\ 
			N(t)&=\begin{bmatrix}   \label{al:2}
				0 & 0 & 0 & 0 \\
				1 & 0 & 0 & 0 \\
				t & 1 & 0 & 0 \\
				0 & t &1 & 0
			\end{bmatrix}. 
		\end{align}
		
		{\rm (a)} If $b\neq0$, then there exists a $t_0\in \F$ such that for $N:=N(t_0)$ we have $\alpha(NA(a,b)N)=\delta(NA(a,b)N) =3$ and $\alpha(NB(a,b,w)N)=\delta(NA(a,b,w)N) =2$.
		
		{\rm (b)} Taking $M=I+N(0)$ gives that $\alpha(MA(0,0)M=\delta(MA(0,0)M)=3$ and, $MB(0,0,w)M$ is invertible.
		
	\end{lemma}
	\begin{proof}
		If $b\neq 0$, the assertions follow from
		
		$$NA(a,b)N=
		\begin{bmatrix}
			0 & 0 & 0 & 0 \\
			b & 0 & 0 & 0 \\
			a+b t_0 & 0 & 0 & 0 \\
			at_0 & bt_0 & b & 0 \\
		\end{bmatrix},
		\ \ \ 
		(NA(a,b)N)^2=b(a+2bt_0)E_{41},$$
		and
		$$NB(a,b,w)N=\begin{bmatrix}
			0 & 0 & 0 & 0 \\
			b & 0 & 0 & 0 \\
			a+b t_0+u & 0 & 0 & 0 \\
			t_0 (2u +a) & b t_0+u & b & 0 \\
		\end{bmatrix}, \ \ \ 
		(NB(a,b,w)N)^2=b(a+2bt_0+2w)E_{41}.$$
		There is only one $t_0$ such that  $a+2bt_0+2w=0$. Note that hence $a+2bt_0\neq 0$. Now apply Lemma \ref{lem: basic} to confirm the claims when $b\neq 0$. 
		
		When $a=b=0$, 
		$$MA(0,0)M=\begin{bmatrix}
			0 & 0 & 0 & 0 \\
			1 & 0 & 0 & 0 \\
			1 & 0 & 0 & 0 \\
			0 & 1 & 1 & 0 \\
		\end{bmatrix}$$
		is a nilpotent matrix of nilindex three and, 
		$$MB(0,0,w)M=\begin{bmatrix}
			w & 0 & 0 & 0 \\
			1+2w & w & 0 & 0 \\
			1+w & 2w & w & 0 \\
			0 & 1+w & 1+2w & w \\
		\end{bmatrix}$$
		is invertible since $w\neq 0$.
	\end{proof}
	
	\begin{lemma}\label{lem:tech}
		Let $A\in \bx$ be an algebraic operator of degree two and of rank greater than one. 
		
		$\mathrm(i)$ There exists an  operator $N\in \bx$, such that 
		$\alpha(NAN)=\delta(NAN)\notin \{1,2\} $ or,  $\alpha(ANA)=\delta(ANA)\notin \{1,2\}$. 
		
		$\mathrm(ii)$ If $B=w I + A$ for some scalar $w\neq 0$, then there exists an algebraic operator $N$ of degree at least three such that $\alpha(NAN)\neq \alpha (NBN)$ and  $\delta(NAN)\neq \delta (NBN)$ or, $\alpha(ANA)\neq \alpha(BNB)$ and $\delta(ANA) \neq \delta(BNB)$.
	\end{lemma}
	\begin{proof}
		Let $A^2 = aA +bI$ for some $a,b\in \F$. We consider several cases.
		
		{\sc	Case 1}. Assume $b=0$ and $a\neq 0$.
		
		Now, $A=aP$ for some idempotent $P$ of rank at least two. In the proper basis we can write $A=a \mathrm{diag}(1,1,0) \oplus A_1$. Applying Lemma \ref{lem:list} and setting $N=T_a\oplus 0$  we get that $NAN$ is nilpotent of nilindex three and so, $\alpha(NAN)=\delta(NAN)=3$ so (i) holds. 
		
		For (ii) we have $B= \mathrm{diag}(a+w,a+w,w) \oplus B_1$. If $a+w\neq 0$ then $B= (a+w)\mathrm{diag}(1,1,\frac{w}{a+w}) \oplus B_1$.  As $a,\, w\neq 0$, $\frac{w}{a+w}\notin \{0,1\}$.  Now Lemma \ref{lem:list} applies again with the operator $N$  from the previous paragraph which is algebraic of degree three.  	 
		If $a+w=0$ we take $N_1=(N_0 \oplus 1) \oplus0 $ where $N_0$ is any $2\times2$ rank-one nilpotent matrix and so,  $N_1$ is algebraic of degree three.  Then $AN_1A$ is a rank-one nilpotent and so, $\alpha(AN_1A)=\delta(AN_1A)=2$ while $BN_1B$ is a scalar multiple of a rank-one idempotent and so, $\alpha(BN_1B)=\delta(BN_1B)=1$, proving (ii) in this case.
		
		{\sc	Case 2}. Suppose $b=a=0$. 
		Since $\mathrm{rank\,}A >1$, we can assume that in some basis $A=A(0,0) \oplus A_1$, where $A(0,0)\in M_4(\F)$ has been set up in \eqref{al:1}.   Choose  $N=M\oplus 0$, (see (b) of Lemma \ref{lem:13} ) to achieve that    $\alpha(NAN)=\delta(NAN)=3$. Applying (b) of Lemma \ref{lem:13} and the fact that $N$ is algebraic of degree three further gives that $\alpha(NBN)=\delta(NBN)=1$.  
		
		{\sc	Case 3}. Let now $b\neq 0$. 
		Since $A$ is not a scalar operator, we can fix a vector $x$ such that $x$ and $Ax$ are linearly independent.  Fix a vector $y\in \mathcal{X}$, $y\notin \mathrm{span}\{x,Ax\}$, such that the dimension of $\mathcal{Y}:=\mathrm{span}\{x,Ax,y,Ay\}$ is maximal possible. Note that under our assumptions, $3\le \dim \mathcal{Y}\le 4$. Now let $\mathcal{X}=  \mathcal{Y}\oplus \mathcal{Z}$ and write
		\begin{equation}\label{eq:110}
			A= \begin{bmatrix} 
				A_{1} & A_{12} \\ 0 & A_{2}
			\end{bmatrix}
		\end{equation}
		with properly defined $A_1\in M_n(\F)$, $n=\dim  \mathcal{Y}$, and operators $A_{12}\in\mathcal{B(Z,Y)}$,  $A_2\in\mathcal{B(Y)}$.
		
		{\sc	Case 3.1} Assume $\dim \mathcal{X} \ge 4$, $\dim \mathcal{Y} = 4$ and recall $b\neq 0$. 
		
		In the fixed ordered basis $\{x,Ax,y,Ay\}$ operator $A$ can be represented as \eqref{eq:110} with $ A_1=A(a,b)\in M_4(\F) $ introduced in \eqref{al:1}. Applying (a) of Lemma \ref{lem:13} and choosing the nilpotent $N= N(t_0) \oplus 0$ of nilindex four we get that $\alpha(NAN)=\delta(NAN)=3$, and $\alpha(NBN)=\delta(NBN)=2$.

		{\sc	Case 3.2} $\dim \mathcal{Y} = 3$. 
		
		Being algebraic of order two, the minimal polynomial of $A$ either equals $m_A(\lambda)=(\lambda -\lambda_0)^2$ for some $\lambda_0 \neq 0$ or, $m_A(\lambda)=(\lambda -\lambda_1)(\lambda -\lambda_2)$ where both $\lambda_1$, $\lambda_2\in \F$ are distinct nonzero eigenvalues. We consider the first case. By rescaling if necessary, we assume with no loss of generality, that $\lambda_0=1$. Now we can suppose that in the representation \eqref{eq:110} we have
		$A_1=A_0$, where $A_0$ is set up in \eqref{eq:4}. If $w+\lambda_0=w+1 \neq 0$, we may and we do assume that $B=B_0$, by recalling again. Setting $N=T\oplus 0$, where $T$, algebraic of order three, is introduced in \eqref{eq:4}  yields  that $\alpha(ANA)=\delta(ANA)=3$, while $\alpha(BNB)=\delta(BNB)=2$.  If $w=-1$,  then $B^2=0$ and $A=I+B$. Choose $T=(E_{12}+E_{32})\oplus 0$, a nilpotent of nilindex three to get  $TAT=E_{13}\oplus 0$  and $TBT=0$. Then $\alpha(TAT)=\delta(TAT)=2$ and $\alpha(TBT)=\delta(TBT)=1$.
		
		It remains to consider the last case, where $m_A(\lambda)=(\lambda -\lambda_1)(\lambda -\lambda_2)$ with $\lambda_1\neq \lambda_2$, and neither of the eigenvalues being equal to zero. At least one of the corresponding eigenspaces must be of dimension at least two, say $\lambda_1$. There is no loss in generality if we set $\lambda_1=1$. Hence, we can assume that $A$ is of the form  \eqref{eq:110} with $A_1$ of the form $A_0$ in \eqref{eq:3} with $a\neq 1$. Suppose $w+1\neq 0$. Then  $B_1=A_1+wI_3=(w+1)B_0$, where $B_0$ has also been defined in  \eqref{eq:3}. Since $w+1\neq 0$, we can replace $B_1$ by $B_0$. Taking $N=T_a \oplus 0$ leads us to the desired conclusion. In the case $w=-1$, we choose $N=T\oplus 0$, where $T$ has been set up in \eqref{eq:125}. The clearly $\alpha(TAT)=\delta(TAT)=1$ while $\alpha(TBT)=\delta(TBT)=2$ and we are done.
	\end{proof}
	\begin{lemma}\label{lem:0}
		Let $A \in \mathcal{B(X)}$. The following assertions are equivalent.
		\begin{enumerate}
			\item  $A=0$.
			\item  $\alpha (ATA)=1$ and $\alpha (TAT)=1$ for every $T\in \bx$.
			\item  $\delta (ATA)=1$ and $\delta (TAT)=1$ for every $T\in \bx$.
		\end{enumerate}
	\end{lemma}
	
	\begin{proof}
		Obviously, (1) implies (2) and  (3).
		
		To prove that both (2) and (3) imply (1), let $A\ne 0$.
		Assume that there exists an $x$ such that $x$, $Ax$ and $A^2x$ are linearly independent. Then there exists an $f$ such that $f(Ax)\neq 0$ and $f(A^2x)=0$. Setting $T=x \otimes f$ gives that $ATA=Ax\otimes A^\prime f$ is a rank-one nilpotent, thus $\alpha(ATA)=\delta(ATA)=2$.
		
		Otherwise, $x$, $Ax$ and $A^2x$ are linearly dependent for every $x$, and hence, $A$ is algebraic of degree at most $2$. If $A$ is a non-zero scalar operator, we can choose $T$ to be any nilpotent of nilindex two and obtain that $\alpha(ATA)=\delta(ATA)=2$. So, let us further suppose that $A$ is not scalar operator.
		
		If $\rk A=1$, let $A=x\otimes f$. Suppose $f(x)\neq 0$. Let us choose a non-zero $y\in \ker f$. There exists a non-zero $g\in \mathcal{X}^\prime$ such that $g(x)=g(y)=0$. Let $T:=x\otimes g+y\otimes f$ and compute that $TAT=f(x)^2 y\otimes g$. Since $TAT$ is a rank-one nilpotent, we have $\alpha(TAT)=\delta(T)=2$. If $f(x)=0$, just take $T=I$ to see that $\alpha(TAT)=\delta(TAT)=2$.
		
		In the remaining case, when $A$ is algebraic of degree two and of rank greater than one, by setting $T=N$ in Lemma \ref{lem:tech}, one completes the proof of the implications (2) $\Rightarrow$ (1) and (3) $\Rightarrow$ (1). 
		
	\end{proof}	
	\begin{lemma}\label{rk1}
		Let $A\in \bx$ and $A\neq 0$. Then the following is equivalent.
		\begin{enumerate}
			\item $\rk A=1$.
			
			\item For every $T\in \bx$ we have $\alpha(ATA)\in\{1,2\}$ and $\alpha(TAT)\in\{1,2\}$.
			\item For every $T\in \bx$ we have $\delta(ATA)\in\{1,2\}$ and $\delta(TAT)\in\{1,2\}$.
		\end{enumerate}	
		
	\end{lemma}
	\begin{proof}
		Implications (1) to (2) and (1)  to (3) are clear because the ascent/descent of every rank-at-most-one operator is either $1$ or $2$ and, the existence of $T_0$ excludes the zero operator.
		
		Next we simultaneously prove that from either (2) or (3), follows (1). Assume $\mathrm{rank\,}A \neq 1$. 
		If $A=\lambda I$ for some $\lambda\neq 0$, then $A$ is invertible and so, by choosing $T=A$, we have $\alpha(TAT)=\delta(TAT)=0$. Let now $A$ be a non-scalar operator of rank greater than one.
		We will find an operator  $B\in \bx$ such that $\alpha(ABA)\notin\{1,2\}$ or $\alpha(BAB)\notin\{1,2\}$.
		Suppose that there is an $x$ such that $x$, $Ax$ and $A^2x$ are linearly independent. Then there exists a closed subspace $\mathcal{Y}$ such that $\mathcal{X}=  \mathrm{span}\{x,Ax,A^2x\} \oplus \mathcal{Y}$.  We define an operator $T$ by $Tx=0$, $TAx=x$ and $TA^2x=Ax$ and, $Ty=0$ for all $y\in \mathcal{Y}$.  We easily compute that $TAT$ is nilpotent of nilindex three and thus,  $\alpha(TAT)=\delta(TAT)=3$.
		
		In the case when for every $x\neq 0$, the vectors $x$, $Ax$ and $A^2x$ are linearly dependent, $A$ is an algebraic operator of degree  $2$ since $A$ is not a scalar operator. Since the rank of $A$ is greater than one, by Lemma \ref{lem:tech} there exists an operator $N$ such that $\alpha(NAN)= \delta(NAN)\notin \{1,2\}$ or, $\alpha(ANA)=\delta(ANA)\notin \{1,2\}$, neither (2) nor (3) are valid.
		
	\end{proof}
	\begin{lemma}\label{lem:p1}
		Let $A$ be a rank-one operator. Then $A\in \F  \mathcal{P}_1$ if and only if $\alpha(ATA)=1$ for every $T\in \bx$ if and only if $\delta(ATA)=1$ for every $T\in \bx$.
	\end{lemma}
	\begin{proof}
		The sufficient part is clear. Assume now that $A=x\otimes f\notin \F\mathcal{P}_1$. Then $A\neq0$ and $f(x)=0$. There exists an operator $T$ such that $f(Tx)\neq 0$. Then $ATA=f(Tx)x\otimes f$ and $\alpha(ATA)=\delta(ATA) =2$.
	\end{proof}
	\begin{lemma}\label{lem:N1}
		Let $A,B\in \bx$. Then $B=u I +v A$ for some scalars  $u,v$, $v\neq 0$,   if and only if for every $N\in \mathcal{N}_1(\mathcal{X})$ the following holds true:
		\[ NAN\in \mathcal{N}_1(\mathcal{X}) \Longleftrightarrow   NBN\in \mathcal{N}_1(\mathcal{X}). \]
		
	\end{lemma}
	
	\begin{proof}
		Let $N=x\otimes g$, $g(x)=0$. Then $NAN=g(Ax)N$ and $NBN=g(Bx)N$. Assume that there exists an $x$ such that $Bx \notin \langle x, Ax \rangle$. Then, we can find a $g$ such that $g(x)=g(Ax)=0$ and  $g(Bx)\ne 0$ giving that $NAN\notin \mathcal{N}_1(\mathcal{X})$ and $NBN\in \mathcal{N}_1(\mathcal{X})$. Otherwise, we have $Bx \in \langle x, Ax \rangle$ for every $x$ and by applying Lemma 2.4 in \cite{9}, we have $B=u I +v A$ for some non-zero $v$.
	\end{proof}
	The following is a straightforward consequence of the previous Lemma.
	\begin{corollary}\label{cor1}
		Let $M$ and $K$ be rank-one nilpotent operators. Then $M$ and $K$ are linearly dependent if and only if the following condition is satisfied.
		$$NMN\in \mathcal{N}_1(\mathcal{X}) \Longleftrightarrow   NKN\in \mathcal{N}_1(\mathcal{X})~~~\ \ \ ( N \in \mathcal{N}_1(\mathcal{X})).$$
	\end{corollary}
	
	\begin{proposition}\label{prop:sim}
		Let $M,N\in\mathcal{N}_1(\mathcal{X})$ be linearly independent. Then the following
		are equivalent:
		\begin{itemize}
			\item[(1)]	$M \sim N$.
			\item[(2)]  There exists a $B \in \mathcal{N}_1(\mathcal{X})$ such that $B$ is neither a scalar multiple of $M$ nor is a scalar multiple of $N$  and, for
			every $T \in \mathcal{B(X)}$, we have
			\begin{equation}
				MTM\notin \mathcal{N}_1(\mathcal{X}) \text{ and } NTN\notin \mathcal{N}_1(\mathcal{X}) \Rightarrow  BTB\notin \mathcal{N}_1(\mathcal{X}).
			\end{equation}
		\end{itemize}
	\end{proposition}
	\begin{proof}
		Suppose that $M$ and $N$ have the same range. Then, we can write $M=x\otimes f$ and $N=x\otimes g$, $f(x)=g(x)=0$, for some linearly independent functionals  $f$ and  $g$.  Set $B:=M + N\in \mathcal{N}_1(\mathcal{X})$ and compute $MTM = f(Tx)M$ and $NTN=g(Tx)N $. As $M\in \mathcal{N}_1(\mathcal{X})$, we have	$MTM,\ NTN\notin \mathcal{N}_1(\mathcal{X}) $ if and only if $f(Tx)=g(Tx)=0$. Then  $BTB=(f+g)(Tx)B\notin  \mathcal{N}_1(\mathcal{X})$ since $(f+g)(Tx)=0$.
		
		In the case when  $M$ and $N$ have the same kernel, the same $B$ would do the job.
		
		To prove the other direction, assume that $M = x\otimes f $ and $N =y \otimes g$, $f(x)=g(y)=0$, are rank-one nilpotents
		such that $x$ and $y$ as well as $f$ and $g$ are linearly independent. We have to show that for every operator $B\in \mathcal{N}_1(\mathcal{X})$ there exists an operator $T\in \mathcal{B(X)}$ such that 	$MTM\notin \mathcal{N}_1(\mathcal{X})$ and $NTN\notin \mathcal{N}_1(\mathcal{X})$ but $BTB\in \mathcal{N}_1(\mathcal{X})$.
		
		Fix a rank-one nilpotent $B=z\otimes h$, $h(z)=0$. It will suffice to define an operator $T$ such that $f(Tx)=g(Ty)=0$ and $h(Tz)\neq 0$.
		
		Suppose that $x,y,z$ are linearly independent. We can always find a nonzero vector $x_0\in\ker f \cap \ker g$, since $\dim \mathcal{X}\ge 3$, and, we choose a vector $z_0$ such that $h(z_0)\neq0$.  Then we can define an operator $T$ such that $Tx=x_0$, $Ty=x_0$ and $Tz=z_0$.
		
		Let now $z=\beta x+\gamma y$ for some $\beta,\gamma$ not both equal to zero and $h\notin \langle f,g\rangle$. Then there exist $x_0,y_0$ such that $f(x_0)=g(y_0)=0$, $h(x_0)=\overline{\beta}$ and  $h(y_0)=\overline{\gamma}$. Define $Tx=x_0$, $Ty=y_0$. We have $f(Tx)=g(Ty)=0$ and   $h(Tz)=\beta h(Tx) +\gamma h(Ty)=|\beta|^2+|\gamma|^2 \neq 0$.
		
		Finally, let  $z=\beta x+\gamma y$ and $h=\mu f +\nu g$. Since $B=z\otimes h$ is neither a scalar multiple of $N$ nor a multiple of  $M$, $z\neq 0$ and $h\neq 0$, we infer that  $\gamma \mu$ and $\beta \nu$ cannot be both equal to zero. Now, there exist $x_0,y_0$ such that $f(x_0)=g(y_0)=0$ and $g(x_0)=\overline{\beta \nu}$ and  $f(y_0)=\overline{\gamma \mu}$. Set $Tx=x_0$ and $Ty=y_0$. Then, $f(Tx)=g(Ty)=0$ but
		\begin{align*}
			h(Tz) &=(\mu f +\nu g)(\beta Tx+\gamma Ty )\\
			&=(\mu f +\nu g)(\beta x_0+\gamma y_0 )\\
			&=\gamma\mu f(y_0) +\beta \nu g(x_0) \\
			&=|\gamma\mu|^2 +|\beta \nu|^2 \neq 0.
		\end{align*}
	\end{proof}
	
	The previous result has provided a sort of geometrical structure on the set of rank-one nilpotents, which can be further used to determine the form of the restriction of $\phi$ to the set of at most rank-one nilpotents.

	\begin{proposition}\label{prop:DHB}
		Let $\mathcal{X}$ be a complex  Banach space of dimension at least three. Suppose that $\psi:\mathcal{N}_1(\mathcal{X})\cup \{0\}\to \mathcal{N}_1(\mathcal{X})\cup \{0\}$ is a surjective mapping  satisfying $\psi(0)=0$, preserving linear dependency in both directions and
		\[ T\sim S  \Longleftrightarrow \psi(T)\sim \psi(S) \]
		for all $T$, $S \in \mathcal{N}_1(\mathcal{X})$.
		
		If $\mathcal{X}$ is infinite-dimensional space, then  either there exist an invertible bounded linear or conjugate-linear operator $A: \mathcal{X}\to \mathcal{X}$ and a map $\lambda : \mathcal{N}_1(\mathcal{X})\to \F^\ast$ such that
		\[\psi(T)=\lambda(T) ATA^{-1},\ \ \ T\in  \mathcal{N}_1(\mathcal{X}), \]
		or,
		there exists an invertible bounded linear or conjugate-linear operator $A: \mathcal{X}^\ast\to \mathcal{X}$, a map $\lambda : \mathcal{N}_1(\mathcal{X})\to \F^\ast$ such that
		\[\psi(T)=\lambda(T) AT^\prime A^{-1},\ \ \ T\in  \mathcal{N}_1(\mathcal{X}). \]
		If $\mathcal{X}$ is of finite dimension $n\ge 3$, we identify operators with n-by-n matrices. Then  there exists a nonsingular matrix $A$, a field automorphism $\tau$ of $\F$ and a map $\lambda : \mathcal{N}_1(\mathcal{X})\to \F^\ast$ such that $\psi$ is either of the form
		\[ \psi([t_{ij}]) = \lambda(T) A[\tau(t_{ij})]A^{-1},\ \ \ [t_{ij}]\in \mathcal{N}_1(\mathcal{X}),\]
		or, $\phi$ is of the form
		\[ \psi([t_{ij}]) = \lambda(T) A[\tau(t_{ij})]^{\mathrm{tr}}A^{-1},\ \ \ [t_{ij}]\in \mathcal{N}_1(\mathcal{X}).\]
	\end{proposition}
	\begin{remark}
		The proof of the above proposition is very similar to the proof of  Lemma 2.2 in \cite{6}, however, it is not directly applicable in our case, so we present the proof.
	\end{remark}
	\begin{proof}
		For any nonzero $x\in\mathcal{X}$ and $f\in \mathcal{X}^\ast$ define $L_x=\{x\otimes g; g(x)=0\}$ and $R_f=\{y\otimes f; f(y)=0\}$.
		
		{\sc Step 1.} \textit{For every nonzero $x\in\mathcal{X}$ there exists a nonzero vector $y$ such that $\psi(L_x)=L_y$ or there exists an nonzero functional $g$ such that $\psi(L_x)=R_g$. Moreover,  for every nonzero $f\in\mathcal{X}^\ast$ there exists a nonzero vector $y$ such that $\psi(R_f)=L_y$ or there exists a nonzero functional $g$ such that $\psi(R_f)=R_g$.}
		
		Let us fix a rank-one nilpotent $x\otimes f\in L_x$. Then $\psi(x\otimes f)=y\otimes g$ for some non-zero vector $y$ and a functional $g\ne 0$ such that $g(y)=0$. Let  $f_1$ be any functional linearly independent of $f$ and satisfying $f_1(x)=0$.  Since $x\otimes f\sim x\otimes f_1$, we conclude that $\psi(x\otimes f)\sim \psi(x\otimes f_1)$. Then either  $\psi(x\otimes f_1)\in L_y$ or, $\psi(x\otimes f_1)\in R_g$.  If incidentally $\psi(x\otimes f_1)\in L_y\cap  R_g$, then $\psi(x\otimes f_1)=\gamma y\otimes g$ and some nonzero scalar $\gamma$, which contradicts the linear independency of $\psi(x\otimes f)$ and $\psi(x\otimes f_1)$.  For any $f_2\in \{x\}^\circ$ we have $x\otimes f_2 \sim x\otimes f_1\sim x\otimes f_2$. Assuming  $\psi(x\otimes f_1)=y\otimes g_1$ and $\psi(x\otimes f_1)\notin R_g$ provides that $g$ and $g_1$ are linearly independent, hence $\psi(x\otimes f_2)\in L_y$. Similarly, supposing $\psi(x\otimes f_1)=y_1\otimes g$ gives that $y$ and $y_1$ are linearly independent and so, $\psi(x\otimes f_2)\in R_g$. We have thus confirmed that $\psi(L_x)\subseteq L_y$ or, $\psi(L_x)\subseteq R_g$. Next we show that  $\phi(L_x)=L_y$ when $\phi(L_x)\subseteq L_y$. Take any $y\otimes h\in  L_y$. Since $\psi$ is surjective, $y\otimes h= \psi(u\otimes v)$, for some nonzero $u\in\mathcal{X}$, $v\in \mathcal{X}^\ast$. As $y\otimes h\sim y\otimes g=\psi(x\otimes f)$ and $y\otimes h\sim y\otimes g_1=\psi(x\otimes f_1)$, it follows that $u$ and $x$ are linearly dependent, giving that $u\otimes v \in L_x$. In a very similar way, one can show that in fact $\psi(L_x)= R_g$ when $\psi(L_x)\subseteq R_g$.
		
		{\sc Step 2.} \textit{For every nonzero $x\in\mathcal{X}$ we have $\psi(L_x)=L_{y_x}$ for some nonzero vector $y_x$; or, for every nonzero $x\in\mathcal{X}$ we have $\psi(L_x)=R_{g_x}$ for some functional $g_x$. Similarly, for every nonzero $f\in\mathcal{X}^\ast$ we have $\psi(R_f)=L_{y_f}$ for some $y_f$; or, for every nonzero $f\in\mathcal{X}^\ast$ we have $\psi(R_f)=R_{g_f}$ for some  functional $g_f$.}
		
		Assume erroneously that there are linearly independent vectors $x_1$, $x_2$ such that $\psi(L_{x_1})=L_{y_1}$ for some nonzero vector $y_1$ and  $\psi(L_{x_2})=R_{g_2}$ for some functional $g_2\neq 0$.  Take a functional  $f\neq 0$ such that $f(x_1)=f(x_2)=0$ and apply that $\psi(x_1\otimes f)=y_1\otimes g_1$ and $\psi(x_2\otimes f)=y_2\otimes g_2$. Since $x_1\otimes f\sim x_2\otimes f$, we have $y_1\otimes g_1\sim y_2\otimes g_2$, giving that either $y_1,y_2$ are linearly dependent or, $g_1,g_2$ are linearly dependent. If  $y_1,y_2$ are linearly dependent, we have $\psi(x_2\otimes f)=y_2\otimes g_2\in L_{y_1}\cap R_{g_2}=\psi(L_{x_1})\cap \psi(L_{x_2})$. Then $y_2\otimes g_2=\psi(x_1\otimes h_1) =\psi(x_2\otimes h_2)$. It follows that $x_1\otimes h_1$ and $x_2\otimes h_2$ are linearly dependent, forcing that $x_1$ and $x_2$ are linearly dependent, a contradiction. Similarly, we treat the alternative options.
		
		{\sc Step 3.} \textit{For every nonzero $x\in \mathcal{X}$ and every nonzero  $f\in \mathcal{X}^\ast$ there either exist $y_x\in\mathcal{X}$ and  $g_f\in \mathcal{X}^\ast$ such that
			\begin{equation}\label{eq:1}
				\psi(L_x)=L_{y_x} \ \ \text{and} \ \  \psi(R_f)=R_{g_f}
			\end{equation}
			or, there exist $y_f\in\mathcal{X}$ and  $g_x\in \mathcal{X}^\ast$ such that
			\begin{equation}\label{eq:2}
				\psi(L_x)=R_{g_x} \ \ \text{and} \ \  \psi(R_f)=L_{y_f} \ \ \ (x\in \mathcal{X},\   f\in \mathcal{X}^\ast).
			\end{equation}
		}
		
		Assume $ \psi(L_x)=L_{y_x}$ for all $x\neq0$ and  $\psi(R_f)=L_{g_f}$ for every nonzero functional $f$. Choose linearly independent vectors $x_1$ and $x_2$ and a nonzero functional $f$ satisfying $f(x_1)=f(x_2)=0$. Then $\psi(x_1\otimes f)=y_{x_1}\otimes g_1=y_f\otimes h_1$ and $\psi(x_2\otimes f)=y_{x_2}\otimes g_2=y_f\otimes h_2$. It follows that $y_{x_1}$ and  $y_{x_2}$ are linearly independent, thus $\psi(L_{x_1})=L_{y_{x_1}}=L_{y_{x_2}}=\psi(L_{x_2})$. For every $f\in\{x_1\}^\circ$ there exists a $g\in\{x_2\}^\circ$ such that $\psi(x_1\otimes f)=\psi(x_2\otimes g)$. It follows that $x_1\otimes f$ and $x_2\otimes g$ are linearly dependent, a contradiction with linear independence of $x_1$, $x_2$.

		Assuming that the $\psi$ satisfies \eqref{eq:1}, we introduce the map $\varphi:  \mathbb{P}(\mathcal{X})\to \mathbb{P}(\mathcal{X})$ by $\varphi([x])=[y_x]$. If $\psi$ has  property \eqref{eq:2}, then we set up the map $\varphi: \mathbb{P}(\mathcal{X})\to \mathbb{P}(\mathcal{X}^\ast)$ by $\varphi([x])=[g_x]$.
		
		{\sc Step 4.} \textit{In each of the above cases $\varphi$ is bijective and, $[x]\subseteq[u]+[v]$ if and only if $\varphi([x])\subseteq\varphi([u])+\varphi([v])$.}
		
		Let $\psi$ satisfy \eqref{eq:1}. We check that $\varphi $ is bijective. Assume first that $\left[y_x\right]=\left[y_z\right]$. It follows that $\psi(L_x)=\psi(L_z)$ thus for every $f\in \left\{x\right\}^\circ $ there exists a $g\in \left\{z\right\}^\circ $ such that $\psi(x\otimes f)=\psi(z\otimes f)$. Then $x\otimes f$ and $z\otimes f$ are linearly dependent yielding that $x$ and $z$ are linearly dependent. Therefore, $\left[x\right]=\left[z\right]$ and in turn $\psi$ injective. The surjectivity of $\varphi$ follows easily from the surjectivity of $\psi$.
		
		Let $\psi(L_u)=L_{y_u}$, $\psi(L_v)=L_{y_v}$ and $\psi(L_x)=L_{y_x}$ and, suppose $[x]\subseteq[u]+[v]$. Take any $g\in   \left\{ y_u, y_v\right\}^\circ$. We will show that  $g(y_x)=0$ which confirms $[y_x]\subseteq[y_u]+[y_v]$. Let (possibly not uniquely defined) $M,N$ be rank-one nilpotents such that $\psi(N)=y_u\otimes g$ and $\psi(M)=y_v\otimes g$. Since $y_u\otimes g$, $y_v\otimes g\in R_g=\psi(R_f)$ for some $f\in \left\{u,v\right\}^\circ$, then by our assumption $f\in \left\{x\right\}^\circ$, $x\otimes f\in L_x \cap R_f$ provides that $\psi(x\otimes f)\in L_{y_x} \cap R_g$. Therefore, $g(y_x)=0$. Conversely, assume $[y_x]\subseteq [y_u]+[y_v]$ and take an $f\in \left\{u,v\right\}^\circ$. It suffices to see that $f(x)=0$. Since $\psi(u\otimes f)$, $\psi(v\otimes f)\in L_{y_u}\cap R_{g_f}$ and $L_{y_v}\cap R_{g_f}$, respectively, for some functional $g_f$, we must have $g_f(y_x)=0$, so, $y_x\otimes g_f\in \psi(L_x)\cap \psi(R_f)$. Let $A$ be a rank-one nilpotent such that $\psi(A)=y_x\otimes g_f$. $A$ is defined up to a scalar factor, because of linear dependency preserving property, so $A$ must be a member of $L_x \cap R_f$ and as it is nilpotent, we have $f(x)=0$.
		
		If  $\psi$ satisfies \eqref{eq:2}, the arguments are very similar and we omit this part of the proof.

		{\sc Step 5.}\textit{ $\psi$ is of one of the forms in the claim.}
		
		The verification of this step is identical to the proof of Claim 4 in the proof of Lemma 2.2 in \cite{6} so we omit it.
	\end{proof}

	\section{Proof of the main result}
	
	Applying Lemma \ref{lem:0}, Lemma \ref{rk1} and Lemma \ref{lem:p1} we immediately obtain that $\phi(A)=0$ if and only if $A=0$, $ \phi$ preserves rank-one operators in both directions and, $\phi$ preserves nonzero scalar multiples of rank-one idempotents and consequently, $\phi$ preserves rank-one nilpotent operators in both directions.
	
	Note that for every rank-at-most-one operator $N\in \bx$ the following observation holds:  $N \in \mathcal{N}_1(\mathcal{X})$ if and only if $\alpha(N)=2$ if and only if $\delta(N)=2$. 
	Having this in mind, by Corollary \ref{cor1}, the restriction of $\phi$ to $\mathcal{N}_1(\mathcal{X} )$, $\psi:\mathcal{N}_1(\mathcal{X} )\to \mathcal{N}_1(\mathcal{X} )$ preserves linear dependence in both directions and, by Proposition \ref{prop:sim}, the map $\psi$ preserves the relation $\sim$ in both directions.
	
	Without losing any generality, by  Proposition \ref{prop:DHB}, we may and we do assume that $\phi(N)=N$ for every rank-one nilpotent $N$. 
	Taking into account Lemma \ref{lem:N1} and the above observation, with an arbitrary $A\in \bx$ and  $B:=\phi(A)$, we have that $NAN\in  \mathcal{N}_1(\mathcal{X} )$ if and only if $\alpha(NAN)=2$ if and only if  $\delta(NAN)=2$ if and only if $NBN\in  \mathcal{N}_1(\mathcal{X} )$. Hence, we get that $B=u_A I +v_A A$ for some  scalars $u_A$, $v_A \in \F$, not necessary unique. Also, 
	by Lemma \ref{lem:N1}, $\phi$ maps scalar operators to scalar operators and vice versa. Then $v_A \neq 0$ whenever $A$ is not scalar operator and by Lemma \ref{rk1},  
	$u_A=0$ for every rank-one operator $A$. 
	
	$A$ is injective (resp. surjective) if and only if $B$ is injective (resp. surjective) subject to preserving ascent (resp. descent). This is due to preserving scalar operators and $\alpha(IAI)=\alpha(A)$ and $\delta(IAI)=\delta(A)$ for every $A$.
	
	Our next task is to show that $u_A=0$ for every $A$ (except for scalar operators where it is not uniquely defined). Without losing any generality we assume that $B= wI +A$ and we will see that $w=0$ for any non-scalar operator $A$ which is of rank at least two.
	
	Suppose that there exists an $x$ such that $x$, $Ax$ and $A^2x$ are linearly independent. We can take a functional $f$ such that $f(x)=0$, $f(Ax)=1$ and $f(A^2x)=0$. Clearly, $Ax\neq0$ and $A^\prime f\neq 0$, so $Ax\otimes fA\in \mathcal{N}_1(\mathcal{X})$ and hence,  $Bx\otimes fB\in \mathcal{N}_1(\mathcal{X})$. As $B^2=w^2 I +  2wAx +A^2x$, one obtains  $0=f(B^2 x)=2w$. In particular, from this argument, we can deduce that $\phi(T) = v_T T$, $v_T\in \F^\ast$, for every algebraic operator $T$ of degree greater than or equal $3$. 
	
	Next, let  $A$ be an algebraic operator of degree $2$ and of rank greater than one.    Now Lemma \ref{lem:tech} leads to contradiction unless $w=0$. Hereby the proof of Theorem \ref{th:1} has been closed.

	\bigskip
	\textbf{Aknowledgement.} The second author is supported by the Slovenian Research Agency
	(core research program P1-0306). The second author also acknowledges COST (European Cooperation in Science and Technology) actions
	CA15140 ( Improving Applicability of Nature-Inspired Optimisation by Joining Theory and Practice (ImAppNIO)) and IC1406 (High-Performance Modelling and Simulation for Big Data Applications (cHiPSet)) and Tomas Bata University in Zl\'in for accessing Wolfram Mathematica.

\end{document}